\documentclass[12pt,a4paper]{article}
\usepackage{amsmath,amsthm,amsfonts,amssymb,enumitem,mathtools,comment,xcolor}
\usepackage{tikz}
\usetikzlibrary{arrows.meta}
\usepackage[capitalize]{cleveref}
\usepackage{soul}
\usepackage{authblk}
\title{On a Conjecture of Las Vergnas }
\author[1]{Steven D. Noble}
\author[2]{Gordon F. Royle}
\affil[1]{School of Computer Science, University of Leeds, Leeds, West Yorkshire, United Kingdom}
\affil[2]{School of Mathematics and Statistics, University of Western Australia, Crawley, WA, Australia}

\newcommand\blfootnote[1]{%
  \begingroup
  \renewcommand\thefootnote{}\footnote{#1}%
  \addtocounter{footnote}{-1}%
  \endgroup
}

\newcommand{\bG}{\mathbb G}

\DeclareMathOperator{\del}{del}
\DeclareMathOperator{\con}{con}
\newtheorem{theorem}{Theorem}
\newtheorem{lemma}[theorem]{Lemma}
\newtheorem{conjecture}[theorem]{Conjecture}
\newtheorem{proposition}[theorem]{Proposition}
\newtheorem{corollary}[theorem]{Corollary}
\theoremstyle{definition}
\newtheorem{definition}[theorem]{Definition}
\newtheorem{remark}[theorem]{Remark}

\newcommand{\divides}{\mid}
\newcommand{\meddivides}{\mathrel{\big|}}

\newcommand{\deltawye}{$\Delta Y$}
\newcommand{\wyedelta}{$Y\!\Delta$\,}
\newcommand{\dwred}{\deltawye\!-reducible\ }

\renewcommand{\leq}{\leqslant}
\renewcommand{\geq}{\geqslant}

\newcommand{\symdiff}{\mathbin{\triangle}}

\begin{document}
\maketitle


\begin{abstract}
In 1988, Las Vergnas conjectured that if $M$ is a binary matroid with bicycle dimension $d$, then for $0 \leq k \leq d$, the $k$th derivative of the diagonal Tutte polynomial $T(M;z,z)$ evaluated at $z=-1$ is an integer multiple of $2^{d-k}$. While this was rapidly disproved for binary matroids and for graphs in general, extensive computations strongly suggested that it might be true for \emph{planar graphs}. In this paper we prove that this is indeed the case. To do this, we consider a stronger divisibility property that we call the LV property, and a larger class of graphs, namely the class of \deltawye-reducible graphs. By a detailed analysis of how a \deltawye exchange affects the coefficients of the diagonal Tutte polynomial, we show that \deltawye-reducible graphs have the LV property. That Las Vergnas' conjecture holds for planar graphs immediately follows because planar graphs are \deltawye reducible and the LV property is stronger than Las Vergnas' divisibility conditions.
\end{abstract}

\section{Introduction}\label{sec:intro}

This paper is concerned with the following conjecture of Las Vergnas \cite{LV33} about the relationship between certain divisibility properties of the coefficients of the diagonal Tutte polynomial $T(z,z)$ and the bicycle dimension of a graph or binary matroid. (Definitions, notation and terminology are given in the next section.) 




\begin{conjecture}\label{conj3}
If $M$ is a binary matroid with bicycle space of dimension $d$, then for $k=0, 1, \ldots, d$ the $k$th derivative of $T(M;z,z)$ evaluated at $z=-1$ is an integer multiple of $2^{d-k}$.
\end{conjecture}


As it happens, this conjecture is not true for binary matroids (the smallest counterexample has rank $5$ and $15$ elements) and it fails even for graphs where $K_8$ is a counterexample. (See \cite{MR4093686} for further details.)
However, the results of extensive computations exploring this conjecture led us to believe that it may still hold for \emph{planar graphs} or a slightly larger class of graphs.

Although the operations of deletion and contraction are fundamental to proving results about both minor-closed classes of graphs and Tutte polynomials, we were unable to give a direct deletion-contraction proof that planar graphs satisfy Conjecture~\ref{conj3}. The most obvious inductive approach fails precisely when every edge of the graph lies in a bicycle because then deleting or contracting any of the edges reduces the bicycle dimension.

To make progress we consider a \emph{stronger} but arguably more natural property than the divisibility conditions of Conjecture~\ref{conj3}. If we let $D(z) = T(G;z,z)$ denote the diagonal Tutte polynomial then the Taylor series of $D$ around $z=-1$ is given by 
\begin{align*}
D(z) &= D(-1) + D'(-1) (z+1) +  \frac{D''(-1)}{2} (z+1)^2 + \\
& \hspace{2cm} \cdots + \frac{D^{(k)}(-1)}{k!} (z+1)^k + \cdots \\
&= b_0 + b_1 (z+1)  + b_2 (z+1)^2 + \cdots.
\end{align*}

Then $b_0 = \pm 2^d$ and the divisibility conditions of Conjecture~\ref{conj3} are satisfied if and only if
\[2^{d-k} \mid k!\, b_k \  \text{ for all }\  k \leqslant d.
\]

Many graphs satisfy these conditions due to powers of $2$ occurring in the difficult-to-manage $k!$ term rather than directly from the coefficients of the diagonal Tutte polynomial. Thus we define a \emph{more stringent} divisibility property which we call the \emph{LV property}. 

\begin{definition}\label{lvproperty}
With the notation above, we say that a graph $G$ has the \emph{LV property} if $\ 2^{d-k} \mid b_k$ for all $k \leqslant d$.
\end{definition}

Classifying all graphs with the 
LV property is hopeless; for one, this includes all graphs with an odd number of spanning trees. However, although the LV property itself is not closed under taking minors, we might still seek large minor-closed classes of graphs with this property. 

Repeating our computations on small graphs, but now testing for the stronger LV property, reveals a more precise picture. First, all the planar graphs that we tested still satisfy the stronger LV property. Second, the smallest graphs (by number of edges) that \emph{do not} have the LV property have $15$ edges. There are $7$ such graphs ranging from $K_6$ on six vertices to the Petersen graph on $10$ vertices. In fact, these graphs form the \emph{Petersen family}, which is the complete set of graphs that can be obtained from the Petersen graph by a sequence of \deltawye and \wyedelta exchanges. Robertson, Seymour and Thomas \cite{MR1164063} showed that the Petersen family is the complete set of excluded minors for the minor-closed class of \emph{linklessly-embeddable} graphs and so one might hope that linklessly-embeddable graphs have the LV property. Unfortunately this is not the case---the graph depicted in Figure~\ref{fig:rhombic} does not have the LV property, but it is an apex graph and therefore is linklessly embeddable (see \cite{MR1164063}).

The graphs in the Petersen family are also excluded minors for the class of \deltawye-reducible graphs (although not the complete set for there are billions more \cite{MR2200535}), leading us to investigate how a \deltawye exchange affects the coefficients of the diagonal Tutte polynomial. This turns out to require rather detailed analysis but eventually leads to our main result:


\begin{theorem}\label{thm:main}
A \dwred graph has the LV property.
\end{theorem}

Planar graphs are \deltawye reducible, and the LV property is stronger than Las Vergnas' original divisibility conditions so therefore we have confirmed our original belief that Conjecture~\ref{conj3} holds for planar graphs.

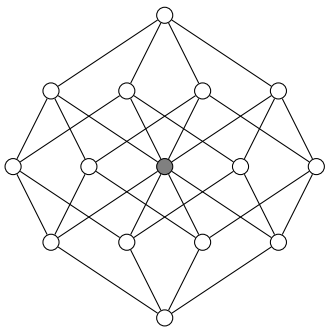
\begin{figure}
\begin{center}
\begin{tikzpicture}
\tikzstyle{vertex}=[circle,draw=black, fill=white, inner sep = 0.75mm]
\tikzstyle{apex}=[circle,draw=black, fill=gray, inner sep = 0.75mm]

\node [vertex] (v0) at (0,2) {};

\node[vertex] (v1) at (-1.5,1) {};
\node[vertex] (v2) at (-0.5,1) {};
\node[vertex] (v3) at (0.5,1) {};
\node[vertex] (v4) at (1.5,1) {};

\node[vertex] (v5) at (-2,0) {};
\node[vertex] (v6) at (-1,0) {};
\node[vertex] (v7) at (1,0) {};
\node[vertex] (v8) at (2,0) {};

\node[vertex] (v9) at (-1.5,-1) {};
\node[vertex] (v10) at (-0.5,-1) {};
\node[vertex] (v11) at (0.5,-1) {};
\node[vertex] (v12) at (1.5,-1) {};

\node [vertex] (v13) at (0,-2) {};
\node [apex] (v14) at (0,0) {};

\draw (v0)--(v1);
\draw (v0)--(v2);
\draw (v0)--(v3);
\draw (v0)--(v4);
\draw (v1)--(v5);
\draw (v1)--(v6);
\draw (v1)--(v14);
\draw (v2)--(v5);
\draw (v2)--(v7);
\draw (v2)--(v14);
\draw (v3)--(v6);
\draw (v3)--(v8);
\draw (v3)--(v14);
\draw (v4)--(v7);
\draw (v4)--(v8);
\draw (v4)--(v14);
\draw (v5)--(v9);
\draw (v5)--(v10);
\draw (v6)--(v9);
\draw (v6)--(v11);
\draw (v7)--(v10);
\draw (v7)--(v12);
\draw (v8)--(v11);
\draw (v8)--(v12);
\draw (v9)--(v13);
\draw (v9)--(v14);
\draw (v10)--(v13);
\draw (v10)--(v14);
\draw (v11)--(v13);
\draw (v11)--(v14);
\draw (v12)--(v13);
\draw (v12)--(v14);

\end{tikzpicture}
\end{center}
\caption{Rhombic dodecahedron plus apex vertex}
\label{fig:rhombic}
\end{figure}


\section{Background}

Our graph-theoretic notation and terminology is mostly standard but for completeness, we briefly outline the main terms in this section.

Let $G=(V,E)$ be a graph with vertex set $V$ and edge set $E$; all our graphs may have loops and or multiple-edges. Let $k(G)$ denote the number of connected components of $G$ and let $r(G) = |V| - k(G)$ denote the \emph{rank} of $G$. For an edge $e$ of $G$, the graphs $G \setminus e$ and $G / e$ are formed by, respectively, \emph{deleting} or \emph{contracting} $e$ from $G$. Parallel edges or loops created during a contraction are retained.

For a subset $A$ of $E$, the graph $G|A$ is the graph $(V,A)$ formed from $G$ by deleting the edges not in $A$ but retaining all the vertices and we let $r_G(A)=r(G|A)$. 
The degree of a vertex $v$ in $G$ is denoted by $d_G(v)$.
 We generally omit the subscript $G$ from $r$, $d$ and similar parameters whenever the context is clear. We say that a connected graph $G$ is \emph{$2$-connected}, if there is no partition of its edge set into non-empty sets $A$ and $B$ satisfying $r(A)+r(B)=r(G)$.

The collection of subsets of $E$ with the operation of symmetric difference forms a binary vector space which has certain subspaces of interest to us.

{

A \emph{cycle} of $G$ is a subset $A \subseteq E$ such that $G|A$ is Eulerian (i.e., every vertex has even degree). Note that this use of the word ``cycle'' is more general than the usual graph-theoretic sense. An edge $e$ is a loop (in the graph theory sense) if and only if $\{e\}$ is a cycle. The symmetric difference of two cycles is a cycle, and the \emph{cycle space} $\mathcal{C}(G)$ of $G$ is the subspace consisting of all the cycles. 

A \emph{cocycle} of $G$ is a subset $A \subseteq E$ such that there is a partition of $V$ into two parts (one possibly empty) such that $A$ is the set of edges with one end in each part. An edge $e$ is a \emph{bridge} in the graph theory sense if and only if $\{e\}$ is a cocycle. The symmetric difference of two cocycles is itself a cocycle and the \emph{cocycle space} $\mathcal{C}^*(G)$ is the subspace consisting of all the cocycles. 

The \emph{bicycle space} $\mathcal{B}(G)$ of $G$ is given by $\mathcal{B}(G) = \mathcal{C}(G) \cap \mathcal{C}^*(G)$ and the elements of this subspace are called \emph{bicycles}. An edge is said to be of \emph{bicycle type} if it is contained in some bicycle. The intersection of a cycle and a cocycle always contains an even number of edges and so a bicycle intersects every cycle and every cocycle in an even number of edges. Finally, the \emph{bicycle dimension} $b(G)$ of $G$ is the dimension of the bicycle space.  See~\cite{G+R} for more information.
}


The \emph{Tutte polynomial} $T(G;x,y)$ of a graph $G=(V,E)$ is defined by
\[ T(G;x,y) = \sum_{A\subseteq E} (x-1)^{r(E)-r(A)}(y-1)^{|A|-r(A)}.\]
See~\cite{E+M} for a comprehensive account of the Tutte polynomial.
Most of this paper is motivated by the remarkable result of Rosenstiehl and Read~\cite{PrincTrip} linking bicycle dimension and the Tutte polynomial.

\begin{theorem}\label{thm:tutte-1-1}
    For every graph $G=(V,E)$,
    \[ T(G;-1,-1) = (-1)^{|E|}(-2)^{b(G)}. \]
\end{theorem}

Putting $z=-1$ into the expression for $D(z)$ we see that $b_0 = (-1)^{|E|} 2^{b(G)}$ which implies the first of the divisibility conditions defining the LV property.




\section{\dwred graphs and knot graphs}

A \deltawye \emph{exchange} in a graph is the operation of deleting three edges $e$, $f$ and $g$ forming a triangle, i.e., a $3$-cycle, and adding a new vertex $w$ together with an edge joining $w$ to each of the three distinct end vertices of $e$, $f$ and $g$. This operation is depicted in Figure~\cref{fig:deltawye}. The inverse operation, that of deleting a vertex of degree $3$ and adding a triangle on its neighbours, is called a \wyedelta exchange.

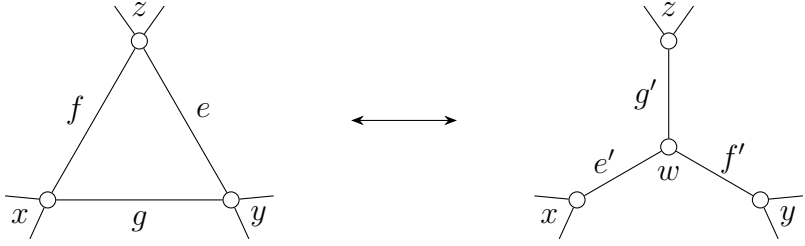
\begin{figure}
\begin{center}
\begin{tikzpicture}[scale=1.4]
\tikzstyle{vertex}=[circle,draw=black,fill=white,inner sep=0.75mm]
\node [vertex] (v0) at (90:1cm) {};
\node [vertex] (v1) at (210:1cm) {};
\node [vertex] (v2) at (330:1cm) {};
\coordinate (vx0) at (80:1.35cm) {};
\coordinate (vx1) at (100:1.35cm) {};
\coordinate (vx2) at (200:1.35cm) {};
\coordinate (vx3) at (220:1.35cm) {};
\coordinate (vx4) at (320:1.35cm) {};
\coordinate (vx5) at (340:1.35cm) {};
\node at (30:0.7cm) {$e$};
\node at (150:0.7cm) {$f$};
\node at (270:0.7cm) {$g$};
\draw (v0)--(v1)--(v2)--(v0);
\draw (v0)--(vx0);
\draw (v0)--(vx1);
\draw (v1)--(vx2);
\draw (v1)--(vx3);
\draw (v2)--(vx4);
\draw (v2)--(vx5);
\node at (210:1.3cm) {$x$};
\node at (330:1.3cm) {$y$};
\node at (90:1.3cm) {$z$};
\begin{scope}[xshift=5cm]
\node [vertex] (v0) at (90:1cm) {};
\node [vertex] (v1) at (210:1cm) {};
\node [vertex] (v2) at (330:1cm) {};
\coordinate (vx0) at (80:1.35cm) {};
\coordinate (vx1) at (100:1.35cm) {};
\coordinate (vx2) at (200:1.35cm) {};
\coordinate (vx3) at (220:1.35cm) {};
\coordinate (vx4) at (320:1.35cm) {};
\coordinate (vx5) at (340:1.35cm) {};
\node [vertex] (c) at (0,0) {};
\node [below] at (c.south){$w$};
\node [above] at (210:0.7cm) {$e'$};
\node [above] at (330:0.7cm) {$f'$};
\node [left] at (90:0.5cm) {$g'$};
\draw (v0)--(c);
\draw (v1)--(c);
\draw (v2)--(c);
\draw (v0)--(vx0);
\draw (v0)--(vx1);
\draw (v1)--(vx2);
\draw (v1)--(vx3);
\draw (v2)--(vx4);
\draw (v2)--(vx5);

\draw[{Stealth[scale=1]}-{Stealth[scale=1]}] (-3,0.25)--(-2,0.25);
\node at (210:1.3cm) {$x$};
\node at (330:1.3cm) {$y$};
\node at (90:1.3cm) {$z$};
\end{scope}
\end{tikzpicture}
\end{center}
\caption{The \deltawye and \wyedelta exchanges}
\label{fig:deltawye}
\end{figure}

We say that graphs $G$ and $H$ are \deltawye \emph{equivalent} if $H$ may be obtained from $G$ by a finite sequence of the following moves and their inverses.
\begin{enumerate}[label=(D\arabic*),itemsep=0.25pt]
    \item Contract a bridge;
    \item Delete a loop;
    \item If edges $e$ and $f$ are parallel, then delete one of $e$ and $f$;
    \item If there is a vertex $v$ of degree two with distinct neighbours, then contract one of the edges incident with  $v$;
    \item A \deltawye exchange.
\end{enumerate}

A graph is \emph{\deltawye reducible} if it is \deltawye equivalent to a graph with no edges. We let $\mathcal D$ denote the class of \dwred graphs. A \emph{\deltawye reduction} of a graph $G$ is a finite sequence $G_0,G_1,\ldots,G_t$ of graphs so that $G_0=G$, $G_t$ is edgeless and for $i=1,\ldots,t$ the graph $G_i$ is obtained from $G_{i-1}$ by one of the moves (D1)--(D5) and their inverses. Notice that in contrast with common practice we do not insist that the number of edges of the sequence of graphs in a \deltawye reduction be weakly decreasing. Nevertheless, the result below
shows that our definition of a \dwred graph coincides with common usage.

\begin{lemma}[{\cite[Lemma~3.2]{knotgraphs}}]
\label{lem:dydec}
Every graph in $\mathcal D$ has a 
\emph{\deltawye reduction}
in which the number of edges is weakly decreasing. 
\end{lemma}

The following result is due to Truemper \cite{Truemper}. 

\begin{lemma}\label{lem:dymin}
The class of \dwred graphs is closed under taking minors.
\end{lemma}

The next theorem was first proved by Epifanov in~\cite{Epifanov}, but see also~\cite{Truemper} for a conceptually very simple proof.
\begin{theorem}
\label{thm:planarisd}
    Every planar graph is \deltawye reducible.
\end{theorem}

We say that graphs $G$ and $H$ are \emph{knot equivalent} if $H$ may be obtained from $G$ by a finite sequence of the following moves and their inverses.
\begin{enumerate}[label=(K\arabic*),itemsep=0.25pt]
    \item Contract a bridge;
    \item Delete a loop;
    \item If edges $e$ and $f$ are parallel, then delete both $e$ and $f$;
    \item If there is a vertex $v$ of degree two with distinct neighbours, then contract both edges incident with  $v$;
    \item A \deltawye exchange.
\end{enumerate}

A graph is a \emph{knot graph} if it is knot equivalent to a graph with no edges. We let $\mathcal K$ denote the class of knot graphs. 
A \emph{$\mathcal K$-reduction} of a graph $G$ is a finite sequence $G_0,G_1,\ldots,G_t$ of graphs so that $G_0=G$, $G_t$ is edgeless and for $i=1,\ldots,t$ the graph $G_i$ is obtained from $G_{i-1}$ by one of the moves (K1)--(K5) and their inverses.

The next two results from~\cite{knotgraphs} show that there is a link between knot graphs and, respectively, bicycle dimension and \dwred graphs. 

\begin{theorem}[{\cite[Theorem~4.1]{knotgraphs}}]
 \label{thm:bpreserved}
    If $G$ and $H$ are knot equivalent, then $b(G)+k(G)=b(H)+k(H)$.
\end{theorem}

The next result is crucial for us. It is a slight strengthening of Proposition~3.1 of~\cite{knotgraphs} and is an immediate consequence of 
Lemmas~\ref{lem:dydec} and~\ref{lem:dymin}.

\begin{proposition}[{\cite[Proposition~3.1]{knotgraphs}}]
\label{prop:dsubk}
    Every graph in $\mathcal D$ has a $\mathcal K$-reduction in which the number of edges is weakly decreasing {and move (K5) or its inverse is only applied to a loopless, bridgeless graph with no parallel edges and no vertices of degree two}. In particular, $\mathcal D \subseteq \mathcal K$. 
\end{proposition}

Combining Theorem~\ref{thm:planarisd} and Proposition~\ref{prop:dsubk} gives the following.
\begin{corollary}
    Every planar graph is a knot graph.
\end{corollary}

In contrast with $\mathcal D$, it is not known whether every graph in $\mathcal K$ has a $\mathcal K$-reduction in which the number of edges is weakly decreasing.
Consequently, it is not at all obvious that there are actually any graphs that are not knot graphs. In a private communication reported in~\cite{knotgraphs}, but without giving any details, D.~Archdeacon, N.~Robertson, P. D.~Seymour, R.~Thomas, and D. L.~Vertigan showed that complete graphs with at least $6$ vertices are not knot graphs. We now present a slight simplification of their argument, which shows that $K_6$ is not a knot graph. Recall that an edge $e$ of $G$ is of \emph{bicycle type} if $e \in B$ for some bicycle $B \in \mathcal{B}(G)$.

\begin{lemma}\label{lem:evenbike}
    Every knot graph has an even number of edges of bicycle type.
\end{lemma}
\begin{proof}
We will show that if $H$ is obtained from $G$ by one of the moves (K1)--(K5) then $H$ has either the same number or two fewer edges of bicycle type than $G$. 

Suppose first that $e$ is a bridge of $G$ and that $H=G/e$. 
As every bicycle is a cycle, $e$ cannot be contained in a bicycle. Moreover a subset $A$ of $E(H)$ is a bicycle in $H$ if and only if it is a bicycle in $G$. Thus $G$ and $H$ have the same number of bicycle edges. A similar argument covers the case where $e$ is a loop of $G$ and $H=G\setminus e$.

Now suppose that $e$ and $f$ are parallel in $G$ and that $H=G\setminus \{e,f\}$. For any partition of the vertices of $G$ into two parts, the endvertices of $e$ belong to different parts if and only if those of $f$ do. Thus every bicycle of $G$ contains either both $e$ and $f$, or neither $e$ or $f$. Moreover, if $B$ is a bicycle of $G$, then $B-\{e,f\}$ is a bicycle of $H$, and if $B$ is a bicycle of $H$, then either $B$ or $B\cup \{e,f\}$ is a bicycle of $G$. Thus $H$ has either the same number or two fewer bicycle edges than $G$. A similar argument covers the case where $H$ is obtained from $G$ by applying (K4).

Finally, suppose that $e$, $f$, $g$ form a triangle in $G$ and $H$ is obtained from $G$ by a \deltawye exchange removing $e$, $f$ and $g$. Label the endvertices of $e$, $f$ and $g$ in $G$ so that $e=yz$, $f=xz$ and $g=xy$. Then $H$ has a new vertex, which we shall label $w$, and new edges $e'=wx$, $f'=wy$ and $g'=wz$; this configuration is depicted in Figure~\ref{fig:deltawye}.

Let $A$ be a subset of $E(G)-\{e,f,g\}$. Then it is straightforward to check that $A$ is a bicycle of $G$ if and only if it is a bicycle of $H$. Recall that a bicycle meets every cycle and every cocycle in an even number of edges. Thus if $B$ is a bicycle in $G$ meeting $\{e,f,g\}$ then $|B\cap \{e,f,g\}|=2$. 
Let $B$ be a bicycle of $G$ and
suppose without loss of generality that $B\cap \{e,f,g\}=\{e,f\}$. Now let $B'=(B-\{e,f\})\cup\{e',f'\}$. Then $d_{H|B'}(z)=d_{G|B}(z)-2$, but every other vertex of $G$ has the same degree in $G|B$ and $H|B'$. Moreover, $d_{H|B'}(w)=2$. Thus $B'\in \mathcal C(H)$. A partition of $V(G)$ into two parts so that $B$ is the set of edges with one endvertex in each part must have $z$ in one part and $x$ and $y$ in the other. Adding $w$ to the part containing $z$ gives a partition of $V(H)$ into two parts so that $B'$ is the set of edges with one endvertex in each part. Thus $B'\in \mathcal C^*(H)$, establishing that $B'$ is a bicycle of $H$. Reversing this argument shows that if $B'$ is a bicycle of $H$ with, say, $B'\cap\{e',f',g'\}=\{e',f'\}$, then $(B'-\{e',f'\})\cup\{e,f\}$ is a bicycle of $G$.  

It follows that every edge of $E(G)-\{e,f,g\}$ has bicycle type in $G$ if and only if it has bicycle type in $H$. Furthermore, for each $x$ in $\{e,f,g\}$, $x$ has bicycle type in $G$ if and only if $x'$ has bicycle type in $H$. So $G$ and $H$ have the same number of edges with bicycle type.

As a graph with no edges has an even number of edges of bicycle type, the proof now follows by induction on the length of a $\mathcal K$-reduction. 
\end{proof}

\begin{theorem}
The complete graph $K_6$ is not a knot graph.
\end{theorem}

\begin{proof}
    By symmetry, either $K_6$ has no edges of bicycle type or every edge has bicycle type. By computing $T(K_6;-1,-1)$ and applying Theorem~\ref{thm:tutte-1-1}, we see that $b(K_6)=4$, so the latter possibility must hold. Thus $K_6$ has $15$ edges of bicycle type and Lemma~\ref{lem:evenbike} implies that it cannot be a knot graph.
\end{proof}

\begin{corollary}
    The class of knot graphs is not closed under minors.
\end{corollary}

\begin{proof}
    Every graph in which each parallel class contains an even number of edges is a knot graph. In particular, the graph formed from $K_6$ by adding an edge in parallel with each edge is a knot graph. But this graph contains $K_6$ as a minor.
\end{proof}

Later we will need one more result which is a special case of a result of Rosenstiehl and Read~\cite{PrincTrip}.

\begin{proposition}\label{prop:bikeminors}
    Let $G$ be a graph with edge $e$. Then $b(G/e)\geq b(G)-1$ and $b(G\setminus e)\geq b(G)-1$.
\end{proposition}

\section{The LV property}

In this section we consider two different expressions for the diagonal Tutte polynomial $T(G;z,z)$, namely
\[
T(G;z,z)= \sum_{k=0}^{|E(G)|} b_k(G)(z+1)^k
\]
and
\[
T(G;z,z)= \sum_{k=0}^{|E(G)|} f_k(G)(z-1)^k.
\]

Here $b_0(G) = \pm 2^{b(G)}$ and $f_0(G)$ is the number of maximal spanning forests of $G$, so if $G$ is connected, then $f_0(G)$ is the number of spanning trees of $G$. Berman~\cite{berman} gave a beautiful algebraic proof that the number of bicycles of a graph divides the number of spanning trees and hence $2^b(G) \divides f_0(G)$.

As described in Section~\ref{sec:intro}, we defined the LV property as a divisibility property on the coefficients $\{b_k(G)\}$. However it turns out to be more convenient to work with the coefficients $\{f_k(G)\}$ and so we need to translate the LV property into a condition on the coefficients $\{f_k(G)\}$. 

\begin{theorem}
If $G$ is a graph with bicycle dimension $b(G)$ then with the notation above, a graph has the LV property if and only if $2^{b(G)-k} \mid f_k(G)$ for all $k \leqslant b(G)$.
\end{theorem}

\begin{proof}
Suppose that $G$ has the LV property as defined by Definition~\ref{lvproperty}. Substituting $z+1 = (z-1)+2$ and expanding each of the terms $((z-1)+2)^i$ using the binomial theorem, we have
\begin{align*}
    T(G;z,z)&= \sum_{i=0}^{|E(G)|} b_i(G) (z+1)^i
    = \sum_{i=0}^{|E(G)|} b_i(G) \sum_{j=0}^i \binom ij (z-1)^j 2^{i-j}\\
    &= \sum_{j=0}^{|E(G)|} \sum_{i=j}^{|E(G)|} \binom ij 2^{i-j} b_i(G) (z-1)^j.
\end{align*}
Thus
\[ f_j(G) = \sum_{i=j}^{|E(G)|} \binom ij 2^{i-j} b_i(G).\]
Choose $j$ with $0\leq j \leq b(G)$. 
Then, as $G$ has the LV property, for $i$ with $j \leq i \leq b(G)$ we have $2^{b(G)-i} \divides b_i(G)$, so $2^{b(G)-j}\divides 2^{i-j}b_i(G)$. And for $i\geq b(G)$ we have $b(G)-j\leq i-j$, so $2^{b(G)-j} \divides 2^{i-j}b_i(G)$. Thus $2^{b(G)-j} \divides f_j(G)$ as required. The reverse argument proceeds very similarly as one may show that
\[ b_j(G) = \sum_{i=j}^{|E(G)|} \binom ij (-2)^{i-j} f_i(G).\qedhere\]
\end{proof}

In other words, the exact same divisibility conditions can be used for either the $\{b_i(G)\}$ or $\{f_i(G)\}$ coefficients. While $b_0(G)$ and $f_0(G)$ can each be interpreted combinatorially, there is no clear interpretation for $b_i(G)$ when $i > 0$. However, the coefficients $\{f_i(G)\}$ do have an interpretation, which we will exploit in the next section to establish that \dwred graphs have the LV property. 
For a graph $G$ let  
\[ \mathcal T(G) = \{A \subseteq E(G) : G|A \text{ is a maximal spanning forest}\}.\] 
So when $G$ is connected $\mathcal T(G)$ is just the collection of (edge sets of) spanning trees of $G$. For a subset $A$ of the edges of $G$, we define its \emph{nullity} in $G$, denoted by $n_G(A)$, to be
\[ n_G(A) = \min_{T \in \mathcal T(G)} |T\symdiff A|.\]
It is straightforward to check that $n_G(A)=r(G)+|A|-2r(A)$.
Thus
\[ T(G;z,z)  = \sum_{A\subseteq E(G)} (z-1)^{n(A)},\]
so for $i=0,\ldots,|E(G)|$ we have
\[ f_i(G) = |\{A\subseteq E(G): n(A)=i\}|.\]
We let 
\[R(G;z)=T(G;z+1,z+1)=\sum_{A\subseteq E(G)} z^{n(A)} = \sum_{i=0}^{|E(G)|} f_i(G)z^i.\]
It is convenient to set $f_i(G)=0$ for $i<0$.

Finally, for use in the next section, we describe how the nullity of a set of edges changes when we make a small change to the set or the graph itself.

\begin{lemma}\label{lem:nullchanges}
Let $G$ be a graph with edge $e$, and let $A$ be a subset of $E(G)$ with $e\notin A$. Then
\[ n_G(A\cup \{e\}) = 
\begin{cases}
    n_G(A) + 1, & \text{if both ends of $e$ lie in one component of $G|A$;}\\
    n_G(A)-1, & \text{otherwise.}
\end{cases}\]
Moreover, if $e$ is not a bridge of $G$, then $n_{G\setminus e}(A)=n_G(A)$ and if $e$ is not a loop of $G$, 
then $n_{G/e}(A)=n_G(A \cup e)$.
\end{lemma}
\begin{proof}
The proof follows easily from the definition of nullity in terms of rank.
\end{proof}

\section{Main result}

This section is devoted to the proof of our main result.
\begin{theorem}
    Every \dwred graph has the LV property.
\end{theorem}

By Proposition~(\ref{prop:dsubk}) any \dwred graph has a $\mathcal K$-reduction where the number of edges of each graph in the sequence is weakly decreasing. Our proof proceeds by showing that a graph in such a $\mathcal K$-reduction has the LV property if and only if the next graph in the sequence does too.  Not surprisingly, the bulk of the work is required when consecutive graphs in a $\mathcal K$-reduction are related by a \deltawye exchange.  

Suppose that $G$ is a graph with 
edges $e$, $f$ and $g$ forming a triangle.  
Let $G_{\del}$ denote the graph obtained from $G$ by deleting $e$, $f$ and $g$. 
Let $H$ denote the graph formed by performing a \deltawye exchange on the triangle $efg$. Let $G_{\con}$ denote the graph formed from $H$ by contracting $e$, $f$ and $g$ and for $p\in \{e,f,g\}$, let $G_p$ denote the graph obtained from $G$ by contracting $p$ and deleting the two edges in $\{e,f,g\}-\{p\}$. 

A key step in our proof is to assume that each of $G_{\del}$, $G_{\con}$, $G_e$, $G_f$ and $G_g$ has the LV property, and show that providing $G_{\del}$ is connected,
this implies that $G$ has the LV property if and only if $H$ does.

To do this we need a lower bound on the bicycle dimension of $G_{\del}$, $G_{\con}$, $G_e$, $G_f$ and $G_g$ in terms of $b(G)$.  
\begin{lemma}\label{lem:bikelower}
With the notation as above, we have $b(H)=b(G)$, $b(G_{\del})\geq b(G)-2$, $b(G_{\con})\geq b(G)-2$ and for every $p$ in $\{e,f,g\}$, $b(G_p)\geq b(G)-1$.
\end{lemma}

\begin{proof}
The fact that $b(H)=b(G)$ follows from Theorem~\ref{thm:bpreserved}. Recall that a bicycle intersects a cycle in an even number of edges. Thus if $B$ is a bicycle in $G$, then $|B\cap\{e,f,g\}|\in \{0,2\}$. So it is not possible to have precisely one of $\{e,f,g\}$ having bicycle type.

When no edge in $\{e,f,g\}$ has bicycle type in $G$, a set $B$ is a bicycle in $G$ if and only if it is a bicycle in each of $G_{\con}$, $G_{\del}$, $G_e$, $G_f$ and $G_g$, so the conclusion holds in this case.

Suppose that two edges in $\{e,f,g\}$ have bicycle type and the other does not. Without loss of generality, we may assume that $e$ and $f$ have bicycle type in $G$, but $g$ does not. 
As every bicycle has even sized intersection with $\{e,f,g\}$, there is a bicycle $B$ of $G$ containing both $e$ and $f$. 
Let $\mathcal Z$ be a basis for $\mathcal B(G)$ containing $B$. Then every element of $\mathcal Z$ has even intersection with $\{e,f\}$, so by replacing each element $B'$ of $\mathcal Z-\{B\}$ by $B\symdiff B'$, if necessary, we see that there is a basis $\{B,B_1,\ldots,B_t\}$ for $\mathcal B(G)$ in which for $i=1,\ldots,t$, $B_i \cap \{e,f,g\}=\emptyset$. 
Then for each $i$, $B_i$ is a bicycle in each of $G_{\con}$, $G_{\del}$, $G_e$, $G_f$ and $G_g$, so we have $b(G_{\con})\geq b(G)-1$, $b(G_{\del})\geq b(G)-1$ and $b(G_p)\geq b(G)-1$ for each $p$ in $\{e,f,g\}$, and the conclusion holds in this case.

Finally, suppose that every edge in $\{e,f,g\}$ has bicycle type. Then, without loss of generality, we may assume that there are bicycles $B_{e,f}$ and $B_{e,g}$ in $G$ with $B_{e,f}\cap\{e,f,g\}=\{e,f\}$ and $B_{e,g}\cap \{e,f,g\}=\{e,g\}$. Let $\mathcal Z$ be a basis for $\mathcal B(G)$ containing $B_{e,f}$ and $B_{e,g}$. Then every element of $\mathcal Z$ has even intersection with $\{e,f,g\}$, so by replacing each element $B$ of $\mathcal Z-\{B_{e,f},B_{e,g}\}$ by $B \symdiff B_{e,f}$, 
$B\symdiff B_{e,g}$, or
$B\symdiff B_{e,f} \symdiff B_{e,g}$, 
if necessary, we see that there is a basis $\{B_{e,f},B_{e,g},B_1,\ldots,B_t\}$ for $\mathcal B(G)$ in which for $i=1,\ldots,t$, $B_i \cap \{e,f,g\}=\emptyset$. 
Then for each $i$, $B_i$ is a bicycle in $G_{\con}$, $G_{\del}$, 
$G_e$, $G_f$ and $G_g$. This gives $b(G_{\con})\geq b(G)-2$ and $b(G_{\del})\geq b(G)-2$. 
Furthermore, 
$B_{e,f}-\{e,f\} \in \mathcal B(G_g)$, $B_{e,g}-\{e,g\} \in \mathcal B(G_f)$ and $(B_{e,f} \symdiff B_{e,g})-\{f,g\} \in \mathcal B(G_e)$, so we see that $b(G_p)\geq b(G)-1$ for each $p$ in $\{e,f,g\}$. This completes the proof.
\end{proof}


Given a subset $A$ of $E(G_{\del})$, the \emph{component type} of $A$ describes the partition of $\{x,y,z\}$ induced by the connected components of $G|A$. We say that $x,y,z$ are \emph{joined by $A$} if they all lie in the same connected component of $G|A$ and are \emph{split by $A$} if they each lie in pairwise distinct connected components of $G|A$. Otherwise we say that $x,y,z$ are \emph{mixed by $A$}.

For $i\geq 0$, we now define 
\begin{align*}
c_i &= |\{A \subseteq E(G_{\del}): n_G(A)=i,\  \text{$x,y,z
$ are joined by $A$}\}|.\\
d_i &= |\{A \subseteq E(G_{\del}): n_G(A)=i+2,\ \text{$x,y,z
$ are split by $A$}\}|.\\
h_i &= |\{A \subseteq E(G_{\del}): n_G(A)=i+1,\  \text{$x,y,z
$ are mixed by $A$}\}|.
\end{align*}
So a set $A$ in which the vertices $y$ and $z$ lie in the same connected component of $G|A$, the vertex $x$ is in a different connected component and 
$n_G(A)=i+1$ contributes to $h_i$. 


It is convenient to set $c_i=d_i=h_i=0$ for $i<0$. We let 
\[ C(z)=\sum_{i=0}^{\infty}  c_i z^i,\
D(z)=\sum_{i=0}^{\infty}  d_i z^i,\
H(z)=\sum_{i=0}^{\infty}  h_i z^i.\]

The next proposition compares the nullities of sets in each of
$G$, $H$, $G_{\del}$, $G_{\con}$, $G_e$, $G_f$ and $G_g$; this will allow us to determine how the coefficients of 
$R(G;z)$
 change after a \deltawye exchange.

\begin{proposition}\label{prop:lotsofnullity}
Let $G$, $H$, $G_{\del}$, $G_{\con}$, $G_e$, $G_f$ and $G_g$ be as described above and suppose that $G_{\del}$ is connected. Let $A\subseteq E(G_{\del})$, $X\subseteq \{e,f,g\}$ and $Y\subseteq \{e',f',g'\}$. Then each of the following hold.
\begin{enumerate}
\item
If $x$, $y$, $z$ are joined by $A$, then 
\begin{align*}
n_G(A\cup X) &= n_G(A) + |X|,\\
n_H(A\cup Y) &= \begin{cases} 
n_G(A) & \text{if $|Y|=1$,}\\
n_G(A) + 1 & \text{if $|Y|\in\{0,2\}$,} \\
n_G(A) + 2 &\text{if $|Y|=3$,}
\end{cases}
\end{align*}
\begin{gather*} n_{G_{\del}}(A)= n_G(A),\
n_{G_{\con}}(A)= n_G(A)+2,\\
n_{G_e}(A)=n_{G_f}(A)=n_{G_g}(A)=n_G(A)+1.\end{gather*}

\item
If $x$, $y$, $z$ are split by $A$, then 
\begin{align*}
n_G(A\cup X) &= \begin{cases} 
n_G(A) & \text{if $|X|=0$,}\\
n_G(A) - 1 & \text{if $|X|\in\{1,3\},$} \\
n_G(A) - 2 &\text{if $|X|=2$,}
\end{cases}\\
n_H(A\cup Y) &= n_G(A)+1 -|Y|,
\end{align*}
\begin{gather*} n_{G_{\del}}(A)= n_G(A),\
n_{G_{\con}}(A)= n_G(A)-2,\\
n_{G_e}(A)=n_{G_f}(A)=n_{G_g}(A)=n_G(A)-1.\end{gather*}
\item 
If $y$ and $z$ are in the same connected component of $G|A$, but $x$ is in a different connected component, then
\begin{align*}
n_G(A\cup X) &= \begin{cases} 
n_G(A)-1 & \text{if $X=\{f\}$ or $X=\{g\}$,}\\
n_G(A) + 1 & \text{if $X=\{e\}$ or $X=\{e,f,g\}$,} \\
n_G(A) &\text{otherwise,}
\end{cases}\\
n_H(A\cup Y) &= \begin{cases} 
n_G(A)-1 & \text{if $Y=\{e',g'\}$ or $Y=\{e',f'\}$,}\\
n_G(A) + 1 & \text{if $Y=\{f',g'\}$ or $Y=\emptyset$,} \\
n_G(A) &\text{otherwise,}
\end{cases}
\end{align*}
\begin{gather*} n_{G_{\del}}(A)= n_G(A),\
n_{G_{\con}}(A)= n_G(A),\\
n_{G_e}(A)=n_G(A)+1,\ 
n_{G_f}(A)=n_{G_g}(A)=n_G(A)-1.\end{gather*}
Analogous results hold if the roles of $x$, $y$ and $z$ are interchanged.
\end{enumerate}
\end{proposition}
\begin{proof}
The result follows from several applications of Lemma~\ref{lem:nullchanges}.
As $G_{\del}$ is assumed to be connected, each of $G$, $H$, $G_{\con}$, $G_e$, $G_f$ and $G_g$ is connected.

All the results for $n_G(A\cup X)$ are obtained by using Lemma~\ref{lem:nullchanges}, starting with $G|A$
and
adding the edges of $X$ one edge at a time, noting whether the ends of the edges added are in the same component just before they are added. 

Next we have $n_H(A)=n_G(A)+1$, because $r(H)=r(G)+1$. The other results for $n_H(A\cup Y)$ are then obtained in a similar way to those for $n_G(A\cup X)$. 

Then $G_{\del}$ may be obtained from $G$ by deleting each of $e$, $f$ and $g$, none of which is a bridge at the point when it is deleted. Thus $n_{G_{\del}}(A)=n_{G}(A)$.
Similarly,
$G_{\con}$ may be obtained from $H$ by contracting each of $e'$, $f'$ and $g'$, none of which is a loop at the point when it is contracted. Thus $n_{G_{\con}}(A)=n_H(A\cup \{e',f',g'\})$.

Finally, for $p$ in $\{e,f,g\}$, we have $n_{G_p}(A) = n_G(A\cup \{e\})$. 
\end{proof}

Using this result, we can now write down expressions for $R(G)$, $R(H)$, $R(G_{\del})$, $R(G_{\con})$ and $R(G_e)+R(G_f)+R(G_g)$ in terms of $C$, $D$ and $H$.
\begin{proposition}
Let $G$, $H$, $G_{\del}$, $G_{\con}$, $G_e$, $G_f$ and $G_g$ be as described above and suppose that $G_{\del}$ is connected. Then
\begin{align*}
R(G;z) &= (z+1)^3 C(z)+(z^2+4z+3) D(z) 
 +2(z+1)^2 H(z).\\
R(H;z) &= (z^2+4z+3) C(z) 
 +(z+1)^3 D(z) 
 +2(z+1)^2 H(z).\\
R(G_{\del};z) &= C(z)+ z^2 D(z)  +zH(z).\\
R(G_{\con};z) &= z^2C(z) + D(z) +zH(z).
\end{align*}
and 
\[
R(G_e;z)+R(G_f;z)+R(G_g;z) = 3z(C(z)+D(z)) + (z^2+2)H(z).
\]
Furthermore,
\begin{gather*}
[z^i]R(G) = c_{i-3} + 3c_{i-2} + 3c_{i-1} + c_i  + d_{i-2} + 4d_{i-1} + 3d_i + 2h_{i-2} + 4h_{i-1} +2h_i.\\
[z^i]R(H) = c_{i-2} + 4c_{i-1} + 3c_i + d_{i-3} + 3d_{i-2} + 3d_{i-1} + d_i + 2h_{i-2} + 4h_{i-1} +2h_i.\\
[z^i](R(H)-R(G)) = (d_{i-3}-c_{i-3})+2(d_{i-2}-c_{i-2})-(d_{i-1}-c_{i-1})-2(d_i-c_i).\\
[z^i]R(G_{\del}) = c_i + d_{i-2} + h_{i-1}.\\
[z^i]R(G_{\con}) = c_{i-2} + d_{i} + h_{i-1}.\\
[z^i](R(G_e)+R(G_f)+R(G_g)) = 3c_{i-1} + 3d_{i-1} + h_{i-2} + 2h_i.
\end{gather*}
    \end{proposition}

\begin{proof}
The first set of expressions may be obtained from Proposition~\ref{prop:lotsofnullity}.  
We have
\begin{align*}
R(G;z) &= \sum_{A\subseteq E} z^{n(A)}\\
&= 
\sum_{\substack{A\subseteq E(G_{\del}):\\
\text{$x,y,z$ are}\\ \text{joined by $A$}}} z^{n(A)} 
\sum_{X\subseteq\{e,f,g\}} z^{n(A\cup X)-n(A)}\\
&\phantom{=}{}+\sum_{\substack{A\subseteq E(G_{\del}):\\
\text{$x,y,z$ are}\\ \text{split by $A$}}} z^{n(A)-2}
\sum_{X\subseteq\{e,f,g\}} z^{n(A\cup X)-n(A)+2}\\
&\phantom{=}{}+\sum_{\substack{A\subseteq E(G_{\del}):\\
\text{$x,y,z$ are}\\ \text{mixed by $A$}}}  z^{n(A)-1}
\sum_{X\subseteq\{e,f,g\}} z^{n(A\cup X)-n(A)+1}\\
&= (z+1)^3C(z) + (z^2+4z+3)D(z) +2(z+1)^2H(z).
\end{align*}
The other expressions may be derived similarly. The second set of expressions follow from the first set by extracting coefficients.
\end{proof}

To simplify the exposition in what follows, in contrast with standard practice, we will write $a \divides b$ to mean $b/a \in \mathbb{Z}$ even when $a$ is not necessarily an integer. 
\begin{theorem}\label{thm:hardbit}
Suppose, in the notation given above, each of 
 $G_{\del}$, $G_{\con}$, $G_e$, $G_f$ and $G_g$ have the LV property and that $G_{\del}$ is connected. Then $G$ has the LV property if and only if $H$ does. 
\end{theorem}

\begin{proof} 
Let $d=b(G)$ and suppose that either $G$ or $H$ has the LV property.  We claim that for all $i\geq 0$, each of the following holds.
\begin{enumerate}
\item $2^{d-i-2}\divides c_i$;
\item $2^{d-i-2}\divides d_i$;
\item $2^{d-i-2} \divides h_i$;
\item $2^{d-i-1} \divides (c_i-d_i)$;
\item $2^{d-i} \meddivides [z^i](R(H)-R(G))$. 
\end{enumerate}
We prove this by induction on $i$.

\textbf{Base case:}
First we consider what happens when $i=0$. By hypothesis and Lemma~\ref{lem:bikelower} 
we have 
$2^{d-2} \meddivides [z^0]R(G_{\con})$,
$2^{d-2} \meddivides [z^0]R(G_{\del})$ and
$2^{d-1} \meddivides [z^0](R(G_e)+R(G_f)+R(G_g))$. 
Thus the first two items hold when $i=0$. We also see that $2^{d-1} \divides 2h_0$, giving the third item. 
Because either $G$ or $H$ have the LV property, it follows that either 
$2^d \divides (c_0+3d_0+2h_0)$ or 
$2^d \divides (3c_0+d_0+2h_0)$.
Combining this with the fact that $2^{d-1}$ 
divides each of $2h_0$, $2c_0$ and $2d_0$, we see that it also divides $c_0-d_0$, establishing the fourth item. Thus $2^d \divides 2(c_0-d_0)=[z^0](R(H)-R(G))$, which gives the fifth item when $i=0$.

\textbf{Inductive step:}
Now suppose that the inductive hypothesis is true when $i\leq k-1$. As all of $G_{\con}$, $G_{\del}$, $G_e$, $G_f$ and $G_g$ have the LV property, we obtain the following three divisibility relations.
\begin{enumerate} 
\item $2^{d-k-2} \meddivides [z^k]R(G_{\con})$,
\item $2^{d-k-2} \meddivides [z^k]R(G_{\del})$ and
\item $2^{d-k-1} \meddivides [z^k](R(G_e)+R(G_f)+R(G_g))$.
\end{enumerate}
From the first and second of these relations we get
\[ 2^{d-k-2} \divides (c_k + d_{k-2} + h_{k-1}) \text{\quad and\quad } 2^{d-k-2} \divides (c_{k-2} + d_{k} + h_{k-1}) .\]
Using the inductive hypothesis, each of $c_{k-2}$, $d_{k-2}$ and $h_{k-1}$ is divisible by $2^{d-k-2}$, so we deduce that
$2^{d-k-2} \divides c_k$ and $2^{d-k-2} \divides d_k$, establishing the first two parts of the inductive hypothesis when $i=k$. 

From the third relation we get
\[  2^{d-k-1} \divides  (3c_{k-1} + 3d_{k-1} + h_{k-2} + 2h_k).\]
By the inductive hypothesis, each of $c_{k-1}$, $d_{k-1}$ and $h_{k-2}$ is divisible by $2^{d-k-1}$, so $2^{d-k-1} \divides 2h_k$ and $2^{d-k-2} \divides h_k$, establishing the third part of the inductive hypothesis when $i=k$. 

As either $G$ or $H$ has the LV property we have either
\begin{align*}
2^{d-k} &\meddivides (c_{k-3} + 3c_{k-2} + 3c_{k-1} + c_k  + d_{k-2} + 4d_{k-1} + 3d_k + 2h_{k-2} + 4h_{k-1} +2h_k)\\
\shortintertext{or}
2^{d-k} &\meddivides (d_{k-3} + 3d_{k-2} + 3d_{k-1} + d_k  + c_{k-2} + 4c_{k-1} + 3c_k + 2h_{k-2} + 4h_{k-1} +2h_k).
\end{align*}
In the former case, using the inductive hypothesis, each of $c_{k-3}$, $c_{k-2}$, $d_{k-2}$, $4d_{k-1}$, $h_{k-2}$ and $4h_{k-1}$ is divisible by $2^{d-k}$, so we deduce that
\[ 2^{d-k} \divides (3c_{k-1} + c_k  + 3d_k +2h_k).\]
From the inductive hypothesis we know that $2^{d-k-1} \divides c_{k-1}$ and we have established that $2^{d-k-1} \divides 2h_k$ and 
$2^{d-k-1} \divides 2d_k$, so we deduce that $2^{d-k-1}$ divides $(3c_{k-1} + c_k  + 3d_k +2h_k)-(3c_{k-1}+4d_k+2h_k)=c_k-d_k$, establishing 
the fourth part of the inductive hypothesis in this case. In the latter case, the same conclusion may be reached by following the same argument but with the roles of $c_i$'s and the $d_i$'s interchanged. Thus the fourth item in the inductive hypothesis holds when $i=k$.  

Finally, the inductive hypothesis implies that 
$2^{d-k}$ divides each of $(d_{k-3}-c_{k-3})$,
$(d_{k-2}-c_{k-2})$ and $(d_{k-1}-c_{k-1})$, and we have just shown that $2^{d-k-1} \divides (d_k-c_k)$. So we deduce that $2^{d-k} \meddivides [z^k] (R(H)-R(G))$, establishing the final part of the inductive hypothesis when $i=k$. By induction, each of the five parts of the claim hold for all $i\geq 0$. 

Theorem~\ref{thm:bpreserved} gives $b(G)=b(H)$, so the result follows from the fifth part of the claim.
\end{proof} 

Our next result deals with the case where two graphs are related via move (K3) or (K4). Recall that the Tutte polynomial satisfies the following recurrence relation. For every graph $G$ with edge $e$:
\[  T(G;x,y) = \begin{cases} 
xT(G/e;x,y), & \text{if $e$ is a bridge of $G$;}\\
yT(G\setminus e;x,y), & \text{if $e$ is a loop of $G$;}\\
T(G/e;x,y)+T(G\setminus e;x,y), &\text{otherwise.}
\end{cases}
\]
Thus
\begin{equation} R(G;z) = \begin{cases} 
(z+1)R(G/e;z), & \text{if $e$ is a bridge of $G$;}\\
(z+1)R(G\setminus e;z), & \text{if $e$ is a loop of $G$;}\\
R(G/e;z)+R(G\setminus e;z), &\text{otherwise.}
\end{cases}
\label{eq:Tutterec} 
\end{equation}


\begin{proposition}\label{prop:K3orK4}
    Suppose that $H$ is obtained from $G$ by applying either move (K3) or (K4), and that $H$ has the LV property. Then $G$ also has the LV property.
\end{proposition}

\begin{proof}
    Suppose that $H$ is obtained from $G$ by deleting two parallel edges $e$ and $f$. There are two cases to consider. 

    First, suppose that $k(G\setminus \{e,f\})=k(G)+1$. Then by Theorem~\ref{thm:bpreserved}, $b(G)=b(G\setminus \{e,f\})+1$. By applying~\eqref{eq:Tutterec}, we have
    \[ R(G;z) = (2z+2) R(G\setminus \{e,f\};z).\]
Hence, for all $i$, we have $f_i(G)=2f_i(H)+2f_{i-1}(H)$. As $2^{b(H)-i} \divides f_i(H)$ and 
$2^{b(H)-(i-1)} \divides f_{i-1}(H)$, we get
$2^{b(H)+1-i} \divides (2f_i(H)+2f_{i-1}(H))$, so $2^{b(G)-i} \divides f_i(G)$, as required.

    Second, suppose that $k(G\setminus \{e,f\})=k(G)$. Then by Theorem~\ref{thm:bpreserved}, $b(G)=b(G\setminus \{e,f\})$. By applying~\eqref{eq:Tutterec}, we have
\[ R(G;z)= R(G\setminus\{e,f\};z) + (z+2)R(G\setminus e/f).\]
Hence, for all $i$, we have 
\[ f_i(G) = f_i(G\setminus\{e,f\}) + f_{i-1}(G\setminus e/f) +2f_i(G\setminus e /f).\]
By Proposition~\ref{prop:bikeminors}, $b(G/f)\geq b(G)-1$. In the proof of Lemma~\ref{lem:evenbike}, we showed that deleting a loop does not change the bicycle space, so 
$b(G\setminus e/f)=b(G/f\setminus e)= b(G/f)\geq 
b(G)-1$. Thus $2^{b(G)}$ divides both $f_i(H)=f_i(G\setminus\{e,f\})$ and $f_{i-1}(G\setminus e/f)$ while $2^{b(G)-1}$ divides $f_i(G\setminus e /f)$ from which we deduce that $2^{b(G)} \divides f_i(G)$.

    The proof of the case where $H$ may be obtained from $G$ by applying move (K4) is almost identical to the second part of the proof for move (K3), just interchanging the roles of contraction and deletion. 
\end{proof}

We can now prove our main result.
\begin{theorem}
    Every \dwred graph has the LV property.
\end{theorem}

\begin{proof}
    The proof proceeds by induction on the number of edges of a \dwred graph. The result is clearly true for graphs with at most one edge. Let $G$ be a \dwred graph with $k>0$ edges.
    Adding or removing isolated vertices to $G$ does not change $R$ or $b$, so we may suppose that $G$ has no isolated vertices. Thus if $G$ is not $2$-connected, then each of its blocks has fewer edges than $G$ and the inductive hypothesis ensures that each block of $G$ has the LV property. As $R(G)$ is multiplicative over blocks and $b(G)$ is additive over blocks, 
    the same is true for $G$ itself. Thus we may assume that $G$ is $2$-connected.
    
{Consider a $\mathcal K$-reduction for $G$ in which the number of edges is weakly decreasing and move (K5) or its inverse is only applied when $G$ has no loops, bridges, parallel edges or vertices of degree two. Such a reduction exists by Proposition~\ref{prop:dsubk}. The reduction begins with a (possibly empty) sequence of \deltawye or \wyedelta exchanges to reach a graph $G'$ with $k$ edges, before applying one of moves (K1)--(K4) giving a graph $G''$ with fewer than $k$ edges.}

{The result of applying a \deltawye or \wyedelta exchange to a $2$-connected graph without vertices of degree two is a $2$-connected graph. 
Thus $G'$ is $2$-connected and has at least two edges. 
Therefore the move used to obtain $G''$ from $G'$ must be either (K3) or (K4).     
    It follows from the inductive hypothesis that $G''$ has the LV property and  Proposition~\ref{prop:K3orK4} implies that $G'$ also has the LV property.}
    
Consider graphs $H$ and $H'$.
If $H'$ is obtained from $H$ by a \deltawye exchange and $H$ is $2$-connected with no vertex of degree two, then, in the notation of Theorem~\ref{thm:hardbit}, $H_{\del}$ is connected. 
Similarly if $H$ is obtained from $H'$ by a \wyedelta exchange and $H'$ is $2$-connected, then $H_{\del}$ is connected. Thus we may apply Theorem~\ref{thm:hardbit} to each of the moves used to obtain $G'$ from $G$ to deduce that $G$ also has the LV property. The result follows by induction.
\end{proof}

Planar graphs are \dwred so the following corollary is immediate.

\begin{corollary}
Every planar graph has the LV property and hence satisfies Las Vergnas' conjecture (Conjecture~\ref{conj3}).
\end{corollary}


\begin{remark}
    The reader may wonder why we have stated Theorem~\ref{thm:main} for \dwred graphs and not for knot graphs. The proof certainly holds for every minor-closed class $\mathcal K'$ of knot graphs having a $\mathcal K$-reduction in which the number of edges is weakly decreasing and every graph appearing in the sequence belongs to $\mathcal K'$. Unfortunately, we do not know of any such class properly including $\mathcal D$. 
\end{remark}

\begin{remark}
Conjecture~\ref{conj3} concerns binary matroids and the reader may wonder whether it is possible to extend the approach we have described from graphs to binary matroids.

As the Fano matroid has bicycle dimension $3$ and $28$ bases, any minor closed family of binary matroids with the LV-property must necessarily comprise only regular matroids.  

Nevertheless, it is worth considering whether the techniques we have 
used may be employed to show that the LV property holds for certain regular but not graphic matroids. 
Moves (D1)--(D4)
may be extended to all matroids, by merely amending (D4) to 
deal with a pair of series elements rather than two edges incident with a vertex of degree two.
For binary matroids, the \deltawye exchange becomes a triangle--triad exchange in which a three element circuit (triangle) is replaced by a three element cocircuit (triad). This move is only permitted when the triangle does not contain a cocircuit and its inverse is only permitted when the triad does not contain a circuit. Following Truemper~\cite{zbMATH00053949}, we then define a binary matroid to be \emph{\deltawye reducible} if it may be transformed to the empty matroid using these moves and their inverses. 

It is straightforward to check that the proof of Theorem~\ref{thm:main} extends to \deltawye reducible matroids, virtually unchanged. 
Clearly, the class of
\deltawye reducible matroids includes the cycle
matroids of all \deltawye reducible graphs. Because the class of \deltawye reducible matroids is closed under duality, it also includes the cocycle matroids of \deltawye reducible graphs. 
Thus the cycle and cocycle matroids of \deltawye reducible graphs satisfy the LV property. 
Interestingly, the regular matroid $R_{10}$ is not 
\deltawye reducible
but has the LV-property.
\end{remark}

\section{A non-extension to embedded graphs}
We close by considering whether 
the LV property may hold for graphs that are cellularly embedded in an orientable surface, or equivalently orientable ribbon graphs. There are several reasons for optimism. First, for an orientable ribbon graph $\mathbb G$ there are natural definitions for both $f_i(\mathbb G)$ and $b(\mathbb G)$ which agree with the definitions for graphs (that are not embedded) when $\mathbb G$ is a plane graph. Second, Berman's Theorem has been extended to orientable ribbon graphs~\cite{MMN}, that is $b(\mathbb G)$ divides $f_0(\mathbb G)$ for every orientable ribbon graph $\mathbb G$. Finally, the authors of~\cite{MMN} suggest that results which hold for planar graphs but not for all graphs might be instances of results for graphs embedded on higher genus surfaces, and present some evidence for this by defining the critical group of an orientable ribbon graph.
Unfortunately, we have found an example showing that such an extension is not possible.

As we are presenting a negative result, we keep our discussion brief and refer the reader to~\cite{Jo-Iain-book} for more information on ribbon graphs. To present the counterexample it suffices to consider orientable bouquets,
which are orientable ribbon graphs with just one vertex.
These may be described by a double-occurrence word, or equivalently a chord diagram. A \emph{double-occurrence word}
is a finite list of symbols in which each symbol appears twice, considered up to cyclic shifts. A \emph{bouquet} is a surface with boundary comprising a collection of closed discs, considered up to homeomorphism. One disc is the vertex and the remaining discs (or ribbons) are pairwise disjoint and are the edges. Each edge  
intersects the vertex in two disjoint arcs, called the \emph{ends} of the edge. A bouquet is \emph{orientable} if for each edge its union with the vertex forms an annulus. To translate between a bouquet and a double-occurrence word, note that a bouquet is determined by the cyclic order in which the ends of the edges appear on the boundary of the vertex. 

A bouquet is a \emph{quasi-tree} if its boundary has only one component. For example, a bouquet with no edges is a quasi-tree; an orientable bouquet with one edge is not a quasi-tree; and an orientable bouquet with two edges is a quasi-tree if and only if the ends of its edges are interlaced, that is, if its double-occurrence word has the form $abab$ but not if it has the form $aabb$. 
More generally, we say that edges $a$ and $b$ are \emph{interlaced} in a bouquet $\mathbb G$ if the ends of $a$ and $b$ are met alternately as the boundary of the vertex of $\mathbb G$ is traversed, or equivalently if $a$ and $b$ are interlaced in the double occurrence word of $\mathbb G$. Although we have not defined ribbon graphs with more than one vertex, the concept of a quasi-tree extends to all ribbon graphs and then a plane graph is a quasi-tree if and only if it is a tree. 

For a bouquet $\mathbb G$ and subset $A$ of its edges, we define the bouquet $\mathbb G|A$ to be the result of deleting the edges of $\mathbb G$ not in $A$. With a slight abuse of notation, a set $A$ of edges of a bouquet is said to be a \emph{quasi-tree} if $\mathbb G|A$ is a quasi-tree and we let $f_0(\mathbb G)$ denote the number of quasi-trees of $\mathbb G$. Again, the concept of a quasi-tree may be defined for ribbon graphs in general, and for plane graphs the number of quasi-trees is equal to the number of spanning trees. Let $\mathcal T(\mathbb G)$ denote the number of quasi-trees of a bouquet $\bG$. Then for a set $A$ of edges of $\bG$, we define its \emph{nullity} $n_{\bG}(A)$ by
\[ n_{\bG}(A) = \min_{T\in \mathcal T(\mathbb G)} |T\symdiff A|,\]
analogously to the earlier definition of nullity for graphs. 
We let $f_i(\mathbb G)$ denote the number of subsets of the edges of $\mathbb G$ having nullity $i$.

Given an orientable bouquet $\mathbb G$, we define its interlacement matrix $M(\mathbb G)$ to be the binary matrix with rows and columns indexed by the edges of $\mathbb G$ having zero diagonal, so that for distinct edges $a$ and $b$, $M(\mathbb G)_{a,b}=1$ if and only if $a$ and $b$ are interlaced.

Bouchet showed~\cite{BouchetPU} that a set $A$ of edges of a bouquet is a quasi-tree if and only if the corresponding principal sub-matrix of $M(\mathbb G)$ is non-singular, with all computations over $GF(2)$. (Here we adopt the convention that the empty matrix is non-singular.) Thus $f_0$, $f_1$, $\ldots$ are all determined by $M(\mathbb G)$, and it is not difficult to show that the nullity of a set of edges coincides with the nullity of the corresponding principal submatrix of $M(\mathbb G)$. Following~\cite{MMN}, we define the bicycle space of a bouquet $\bG$ to be the kernel (over $GF(2)$) of the matrix $I+M(\mathbb G)$ and let $b(\mathbb G)$ denote its dimension. We omit the details here, but it is shown in~\cite{MMN} that by exploiting partial duality, the definition of the bicycle space may be extended to all ribbon graphs and then its definition coincides with the usual one for plane graphs.

We are now ready to present our counterexample. Let $\bG$ be the bouquet having edge set $\{1,2,3,\ldots,10\}$ with double occurrence word
\[1,2,3,4,5,6,1,7,8,9,10,7,5,3,10,2,6,9,4,8.\]				We have
\[ M(\bG) =\begin{pmatrix} 
0&	1&	1&	1&	1&	1&	0&	0&	0&	0\\
1&	0&	0&	1&	0&	1&	0&	1&	1&	0\\
1&	0&	0&	1&	0&	1&	0&	1&	1&	1\\
1&	1&	1&	0&	0&	0&	0&	1&	0&	0\\
1&	0&	0&	0&	0&	1&	0&	1&	1&	1\\
1&	1&	1&	0&	1&	0&	0&	1&	1&	0\\
0&	0&	0&	0&	0&	0&	0&	1&	1&	1\\
0&	1&	1&	1&	1&	1&	1&	0&	0&	0\\
0&	1&	1&	0&	1&	1&	1&	0&	0&	0\\
0&	0&	1&	0&	1&	0&	1&	0&	0&	0
\end{pmatrix}.\]
It can be shown that $b(\bG)=4$, $f_0(\bG)=176$, $f_1(\bG)=396$, $f_2(\bG)=316$ and $f_3(\bG)=115$ which is not divisible by $16/2^3=2$.
The example was found by a computer search, which also reveals that there are no counterexamples with fewer edges.
\bibliographystyle{plain}
\bibliography{lv}

@article{knotgraphs,
 author = {Noble, S. D. and Welsh, D. J. A.},
 title = {Knot graphs},
 fjournal = {Journal of Graph Theory},
 journal = {J. Graph Theory},
 issn = {0364-9024},
 volume = {34},
 number = {1},
 pages = {100--111},
 year = {2000},
 language = {English},
 doi = {10.1002/(SICI)1097-0118(200005)34:1<100::AID-JGT9>3.0.CO;2-R},
 zbMATH = {1462935},
 Zbl = {0946.05045}
}

@article {MR2200535,
    AUTHOR = {Yu, Yaming},
     TITLE = {More forbidden minors for {W}ye-{D}elta-{W}ye reducibility},
   JOURNAL = {Electron. J. Combin.},
  FJOURNAL = {Electronic Journal of Combinatorics},
    VOLUME = {13},
      YEAR = {2006},
    NUMBER = {1},
     PAGES = {Research Paper 7, 15},
      ISSN = {1077-8926},
   MRCLASS = {05C75 (05C83)},
  MRNUMBER = {2200535},
MRREVIEWER = {Isidoro\ Gitler},
       DOI = {10.37236/1033},
       URL = {https://doi.org/10.37236/1033},
}

@article {MR1164063,
    AUTHOR = {Robertson, Neil and Seymour, P. D. and Thomas, Robin},
     TITLE = {Linkless embeddings of graphs in {$3$}-space},
   JOURNAL = {Bull. Amer. Math. Soc. (N.S.)},
  FJOURNAL = {American Mathematical Society. Bulletin. New Series},
    VOLUME = {28},
      YEAR = {1993},
    NUMBER = {1},
     PAGES = {84--89},
      ISSN = {0273-0979,1088-9485},
   MRCLASS = {57M15 (05C10)},
  MRNUMBER = {1164063},
MRREVIEWER = {G.\ Burde},
       DOI = {10.1090/S0273-0979-1993-00335-5},
       URL = {https://doi.org/10.1090/S0273-0979-1993-00335-5},
}

@article {MR4093686,
    AUTHOR = {Brijder, Robert and Hoogeboom, Hendrik Jan},
     TITLE = {Counterexamples to a conjecture of {L}as {V}ergnas},
   JOURNAL = {European J. Combin.},
  FJOURNAL = {European Journal of Combinatorics},
    VOLUME = {89},
      YEAR = {2020},
     PAGES = {103141, 3},
      ISSN = {0195-6698,1095-9971},
   MRCLASS = {05B35},
  MRNUMBER = {4093686},
MRREVIEWER = {Xiangqian\ Zhou},
       DOI = {10.1016/j.ejc.2020.103141},
       URL = {https://doi.org/10.1016/j.ejc.2020.103141},
}

@article {MMN,
    AUTHOR = {Merino, Criel and Moffatt, Iain and Noble, Steven},
     TITLE = {The critical group of a combinatorial map},
   JOURNAL = {Comb. Theory},
  FJOURNAL = {Combinatorial Theory},
    VOLUME = {5},
      YEAR = {2025},
    NUMBER = {3},
     PAGES = {Paper No. 2, 41},
      ISSN = {2766-1334},
   MRCLASS = {05C25 (05C10 05C50 05C57 20K01 91A43)},
  MRNUMBER = {4962116},
}

@book{zbMATH00053949,
 author = {Truemper, K.},
 title = {Matroid decomposition},
 isbn = {0-12-701225-7; 0-9663554-0-7},
 year = {1992},
 publisher = {Boston, MA: Academic Press, Inc.},
 language = {English},
 zbMATH = {53949},
 Zbl = {0760.05001}
}

@book{Jo-Iain-book,
	author = {Ellis-Monaghan, Joanna A. and Moffatt, Iain},
	doi = {10.1007/978-1-4614-6971-1},
	isbn = {978-1-4614-6970-4; 978-1-4614-6971-1},
	mrclass = {05-02 (05C10 05C31 57-01 57M25)},
	mrnumber = {3086663},
	mrreviewer = {Lorenzo Traldi},
	pages = {xii+139},
	publisher = {Springer, New York},
	series = {SpringerBriefs in Mathematics},
	title = {Graphs on surfaces: Dualities, polynomials, and knots},
	year = {2013}}

@article{BouchetPU,
	author = {Bouchet, Andr\'{e}},
	doi = {10.1016/0012-365X(87)90132-4},
	fjournal = {Discrete Mathematics},
	issn = {0012-365X},
	journal = {Discrete Math.},
	mrclass = {05C75},
	mrnumber = {900943},
	mrreviewer = {Hubert de Fraysseix},
	number = {1-2},
	pages = {203--208},
	title = {Unimodularity and circle graphs},
	url = {https://doi.org/10.1016/0012-365X(87)90132-4},
	volume = {66},
	year = {1987}}

@article{berman,
 author = {Berman, Kenneth A.},
 title = {Bicycles and spanning trees},
 fjournal = {SIAM Journal on Algebraic and Discrete Methods},
 journal = {SIAM J. Algebraic Discrete Methods},
 issn = {0196-5212},
 volume = {7},
 pages = {1--12},
 year = {1986},
 language = {English},
 doi = {10.1137/0607001},
 zbMATH = {3943839},
 Zbl = {0588.05016}
}

@article{Epifanov,
 author = {Epifanov, G. V.},
 title = {Reduction of a plane graph to an edge by a star-triangle transformation},
 fjournal = {Soviet Mathematics. Doklady},
 journal = {Sov. Math., Dokl.},
 issn = {0197-6788},
 volume = {7},
 pages = {13--17},
 year = {1966},
 language = {English},
 zbMATH = {3241109},
 Zbl = {0149.21301}
}

@article{Truemper,
 author = {Truemper, K.},
 title = {On the {Delta}-{Wye} reduction for planar graphs},
 fjournal = {Journal of Graph Theory},
 journal = {J. Graph Theory},
 issn = {0364-9024},
 volume = {13},
 number = {2},
 pages = {141--147},
 year = {1989},
 language = {English},
 doi = {10.1002/jgt.3190130202},
 zbMATH = {4108782},
 Zbl = {0677.05020}
}

@book{G+R,
 author = {Godsil, Chris and Royle, Gordon},
 title = {Algebraic {G}raph {T}heory},
 fseries = {Graduate Texts in Mathematics},
 series = {Grad. Texts Math.},
 issn = {0072-5285},
 volume = {207},
 isbn = {0-387-95241-1; 0-387-95220-9},
 year = {2001},
 publisher = {New York, NY: Springer},
 language = {English},
 zbMATH = {1600999},
 Zbl = {0968.05002}
}

@book{E+M,
 editor = {Ellis-Monaghan, Joanna A. and Moffatt, Iain},
 title = {Handbook of the {Tutte} polynomial and related topics},
 isbn = {978-1-4822-4062-7; 978-1-032-23193-8; 978-0-429-16161-2},
 year = {2022},
 publisher = {Boca Raton, FL: CRC Press},
 language = {English},
 doi = {10.1201/9780429161612},
 zbMATH = {7553843},
 Zbl = {1495.05001}
}

@misc{PrincTrip,
 author = {Rosenstiehl, P. and Read, R. C.},
 title = {On the principal edge tripartition of a graph},
 year = {1978},
 language = {English},
 howpublished = {Ann. {Discrete} {Math}. 3, 195-226},
 doi = {10.1016/S0167-5060(08)70508-9},
 zbMATH = {3608085},
 Zbl = {0392.05059}
}

@article{LV33,
 author = {Las Vergnas, Michel},
 title = {On the evaluation at (3,3) of the {Tutte} polynomial of a graph},
 fjournal = {Journal of Combinatorial Theory. Series B},
 journal = {J. Comb. Theory, Ser. B},
 issn = {0095-8956},
 volume = {45},
 number = {3},
 pages = {367--372},
 year = {1988},
 language = {English},
 doi = {10.1016/0095-8956(88)90079-2},
 zbMATH = {4103089},
 Zbl = {0674.05024}
}

\section*{Acknowledgement}

This work was started during the Newton Institute (University of Cambridge) Programme on ``Combinatorics and Statistical Mechanics'' in 2008 and was completed while both authors were visiting Emeric Gioan at LIRMM (University of Montpellier) in 2024. We thank both institutions for their hospitality and Emeric Gioan for useful comments. 
\end{document}